\documentclass[11pt]{amsart}

\usepackage{color,enumitem}
\usepackage{amsmath,amssymb,bbm,bm}
\usepackage{amscd,tikz-cd}
\usepackage{mathrsfs}
\usepackage{amsthm}
\usepackage{mathtools, todonotes}

\theoremstyle{plain}

\newtheorem{theorem}{\bf Theorem}[section]
\newtheorem{lemma}[theorem]{\bf Lemma}
\newtheorem{proposition}[theorem]{\bf Proposition}
\newtheorem{corollary}[theorem]{\bf Corollary}

\theoremstyle{definition}
\newtheorem{definition}[theorem]{\bf Definition}
\newtheorem{example}[theorem]{\bf Example}
\newtheorem{remark}[theorem]{\bf Remark}

\newcommand{\disp}{\displaystyle}

\newcommand{\eqa}[1]{
\begin{align*}
#1
\end{align*}}

\newcommand{\R}{\mathbb{R}}

\newcommand{\N}{\mathbb{N}}

\usepackage[all]{xy}

\title{Common transversals for coset spaces of compact groups}
\author{Hiroshi Ando}
\address{Hiroshi Ando, Department of Mathematics and Informatics, Chiba University, 1-33 Yayoi-cho, Inage, Chiba, 263- 8522,
Japan}
\email{hiroando@math.s.chiba-u.ac.jp}
\author{Andreas Thom}
\address{Andreas Thom, Fakult\"at f\"ur Mathematik, TU Dresden, 01062 Dresden, Germany}
\email{andreas.thom@tu-dresden.de}

\begin{document}

\maketitle
\bibliographystyle{siam}

\begin{abstract}
Let $G$ be a Polish group and let $H \leq G$ be a compact subgroup. We prove that there exists a Borel set $T \subset G$ which is simultaneously a complete set of coset representatives of left and right cosets, provided that a certain index condition is satisfied. Moreover, we prove that this index condition holds provided that $G$ is locally compact and $G/G^\circ$ is compact or $H$ is a compact Lie group. This generalizes a result which is known for discrete groups under various finiteness assumptions, but is known to fail for general inclusions of infinite groups. As an application, we prove that Bohr closed subgroups of countable, discrete groups admit common transversals.
\end{abstract}

\tableofcontents

\section{Introduction}

Let $G$ be a group and $H$ be a subgroup. A set $T \subset G$ is said to be a left transversal for $H$ or a complete set of left coset representatives if
$G = \sqcup_{t\in T} tH$. Similarly $S \subset G$ is said to be a right transversal or a complete set of right coset representatives if $G= \sqcup_{s \in S} Hs$. It is well-known that left transversals need not be right transversals unless $H$ is a normal subgroup. Note however, that if $T$ is a left transveral, then $T^{-1}$ is a right transveral and vice versa. It has been known for a long time that there exists subsets of $G$ which are simultaneously left and right transversals under various finiteness assumptions, for example when $H$ is finite or $[G:H]$ is finite. We call a set which is simultaneously a left and right transversal a common transversal for $H$.

The problem of finding common transversals in groups has a long history dating back more than a hundred years, see \cite{MR3295666} for more information. However, very little has been known regarding the regularity properties of common transversals for topological groups. One exception is the study by Appelgate--Onishi \cite{MR0442099AppelgateOnishi1977}, where it is shown that a closed common transversal exists for profinite groups.

Various classical matrix decomposition theorems can be interpreted as results about the existence of a particularly nice common transversal of a compact subgroup. Indeed, for example every matrix $g \in {\rm GL}(n,\mathbb C)$ can be written uniquely as a product of a unitary matrix $u \in {\rm U}(n)$ and a positive definite matrix $p \in {\rm P}(n,\mathbb C)$, i.e., $g=up$. Similarly, taking inverses, we obtain another decomposition $g=p'u'$. Thus, the set of positive definite matrices is a common transversal of ${\rm U}(n)$ in ${\rm GL}(n,\mathbb C).$
The Iwasawa decomposition (or QR-decomposition in the case of ${\rm GL}(n,\mathbb C)$) provides a similar common transversal, provided by the set of upper triangular matrices whose diagonal entries are positive. 
These decompositions are particularly nice in the sense that the common transversal is closed and satisfies $T^{-1}=T.$ This will not necessarily be true for the common transversals that we construct in this paper.

We prove a corresponding result for $G$ Polish and $H$ a compact subgroup, where the natural requirement on the common transversal also includes a certain regularity. In fact, we will prove that we can take $T$ always to be a Borel subset of $G$. We say that a closed subgroup $H$ of a topological group $G$ satisfies the {\it index condition}, if 
$$[H\colon H\cap xHx^{-1}]=[H\colon H\cap x^{-1}Hx], \quad \forall x \in G.$$ 

It was shown by Ore \cite[Theorem 2.1]{MR0100639Ore1958} that a subgroup $H\le G$ admits a common (set theoretic) transversal if and only if the index condition is satisfied. 
The index condition (in case of finite indices) comes up naturally as a criterion for relative unimodularity of Hecke pairs $(G,H)$ or unimodularity of the corresponding Schlichting completion, see \cite{MR2465930, MR1027069} for more details.

The main results of this paper are the following theorems.

\begin{theorem}\label{thm main}
Let $G$ be a Polish group and $H$ be a compact subgroup. Assume that the index condition is satisfied for $H\le G$. Then there exists a Borel subset of $G$, which is a common transversal for cosets of $H$. 
\end{theorem}

The index condition is satisfied in a variety of situations. Our results on this question are summarized in the following theorem.

\begin{theorem}\label{thm index condition} Let $H\le G$ be as in Theorem \ref{thm main}. Then, the index condition for $H\le G$ holds whenever one of the following conditions are satisfied:
\begin{enumerate}
\item $G$ is an inverse limit of Lie groups, or
\item $H$ is a compact Lie group.
\end{enumerate}
\end{theorem}

\begin{corollary} The index condition for $H \le G$ is satisfied for an arbitrary compact subgroup $H$, whenever $G$ is a locally compact group and $G/G^\circ$ is compact. In particular, this holds when $G$ is compact or when $G$ is connected and locally compact.
\end{corollary}

We want to record that the result is sharp in the sense that the results in Theorem \ref{thm index condition} fails in general if $G$ is locally compact, but not an inverse limit of Lie groups, or $H$ is compact, but not a compact Lie group, see Examples \ref{discrete} and \ref{compact}.

Finally, we return to results about discrete groups. As a consequence of our results on topological groups, we conclude the existence of common transversals for subgroups of countable, discrete groups as long as they are closed in the Bohr topology, see Theorem \ref{bohrtrans}. In particular, this applies to pro-finitely closed subgroups.

\section{The index condition}

In this section, we want to show that the index condition is satisfied in the situations described by Theorem \ref{thm index condition}. Moreover, we explain how the index condition is used in order to find a common transversal of left and right cosets contained in a fixed double coset. This will give a motivation for the proof of the other main theorem.

\begin{proposition} \label{prop same index}
Let $G$ be a Polish group, which is also an inverse limit of Lie groups, and let $H$ be a compact subgroup of $G$. Moreover, let $x\in G$ be arbitrary. Then, $[H\colon H\cap xHx^{-1}]=[H\colon H\cap x^{-1}Hx]\in \N_{\geq 1}\cup \{2^{\omega}\}$. 
\end{proposition}

Let $K \leq H$ be a closed subgroup. Note that since $[H\colon K]$ is the cardinality of the compact metrizable space $H/K$, it is either finite or continuum. 

The following result is a consequence of the famous Gleason-Yamabe theorem \cite{MR58607} and \cite[Lemma 4.5]{MR39730}.

\begin{theorem}[Gleason-Yamabe] Every locally compact group $G$ with $G/G^{\circ}$ compact is an inverse limit of Lie groups.
\end{theorem}

\begin{lemma} \label{reduction to Lie case}
Let $G$ be an inverse limit of Lie groups $G= \lim_{i} G_i,$ with projections $\pi_i \colon G \to G_i$, and $K \leq H$ be a compact subgroups of $G$. The following formula holds. 
\[[H:K]=\lim_{i\to \infty}[\pi_i(H) : \pi_i(K)]\]
\end{lemma}
\begin{proof}
Let $c=[H:K], c_i=[\pi_i(H): \pi_i(K)]\,(i\in \N)$. 
Since $G_i$ is a quotient of $G$ and of $G_{i+1}$, we have $c_i\le c_{i+1}\le c$. 
In particular, the limit $\disp c'=\lim_{i\to \infty}c_i$ exists and satisfies $c'\le c$.

Assume by contradiction that $c'<c$. Because each $c_i$ is then an integer, there exists $j\in \N$ such that $c_i=c'\,(i\ge j)$. 
Then let $\{s_1,\dots,s_{c'}\}\subset H$ be such that its image in $\pi_j(H)$ is the representatives for the left cosets of $\pi_j(K)$ in $\pi_j(H)$. Then we have 
\[\pi_i(H)=\bigsqcup_{k=1}^{c'}\pi_i(s_k)\pi_i(K)\,\,(i\ge j).\]
By $c'<c$, there exists $s_0\in G$ such that $s_0\notin \bigcup_{k=1}^{c'} s_k K$. 
On the other hand, for each $i\ge j$ we have $\pi_i(s_0)\in \bigcup_{k=1}^{c'} \pi_i(s_k K)$. Passing to a subsequence if necessary, we may find $k_0\in \{1,\dots,c'\}$ and $k_i\in K$ such that $\pi_i(s_0)=\pi_i(s_{k_0}k_i)$ for all $i\ge j$, and by using the sequential compactness of $K$, we may further pass to a subsequence to assume that $k_i$ converges to some $k\in K$ as $i\to \infty$. Then $s_0=s_{k_0} k$, contradicting our choice of $s_0$. Thus $c'=c$ holds.
\end{proof}

For a topological space, we denote by $\pi_0(X)$ the set of path connected components of $X$. If $H$ is a topological group, we denote by $H^{\circ}$ the path connected component of the identity element and note that $H^{\circ}$ is a normal subgroup of $H$. If $H$ is a compact Lie group, then $H^{\circ}$ is an open subgroup of $H$. The set $\pi_0(H)$ is naturally identified with $H/H^{\circ}$. In particular, in this case $\pi_0(H)$ is a group in a natural way.

\begin{lemma} \label{Lie case}
Let $G$ be a Polish group, $K \leq H$ be a compact Lie subgroups of $G$. Then,
\begin{enumerate}
    \item $[H\colon K]$ is finite if and only if $\dim(H)=\dim(K)$ holds.
    \item If the latter condition holds, then $H^{\circ}=K^{\circ}$ and the equation $$[H\colon K]=[\pi_0(H)\colon \pi_0(K)]$$ is satisfied.
\end{enumerate}  
\end{lemma}
\begin{proof}
Note that $\dim(H)=\dim(K)$ if and only if the Lie algebras of $H$ and $K$ agree -- and this is equivalent to $H^{\circ}=K^{\circ}$, since the path connected components are generated via the exponential map from the Lie algebra. This implies that $H/K$ is in bijection with $(H/K^{\circ})/(K/K^{\circ}) = \pi_0(H)/\pi_0(K)$ proving that $[H:K]$ is finite and equal to the index of $\pi_0(K)$ in $\pi_0(H)$ in this case.
Conversely, if the index of $K$ in $H$ is finite, then clearly $\dim(H)=\dim(K)$.
\end{proof}

\begin{proof}[Proof of Proposition \ref{prop same index}]
By Lemma \ref{reduction to Lie case} applied to $K=H\cap xHx^{-1}$ and $K=H\cap x^{-1}Hx$, we may assume that $G$ is a Lie group.

Note that the groups $H \cap xHx^{-1}$ and $H \cap x^{-1}Hx$ are topologically isomorphic since they are conjugate by $x \in G.$ In particular, their dimensions agree and Lemma \ref{Lie case} applied to $K = H \cap xHx^{-1}$ and $K=H \cap x^{-1}Hx$ implies that the indices of these groups in $H$ are either both finite or both infinite. In the first case, by the second part of Lemma \ref{Lie case}, the indices are identical with $[\pi_0(H):\pi_0(H \cap xHx^{-1})]$ and $[\pi_0(H):\pi_0(H \cap x^{-1}Hx)]$ respectively.
However, $\pi_0$ of a compact Lie group is finite and $\pi_0(H \cap xHx^{-1})$ is isomorphic to $\pi_0(H \cap x^{-1}Hx)$. Thus, the indices agree and are equal to $\sharp \pi_0(H) / \sharp \pi_0(H \cap x^{-1}Hx)$. This completes the proof.
\end{proof}

\begin{proposition} \label{lem index cpt lie}
Let $G$ be a Polish group and $H$ be a closed subgroup of $G$ which is a compact Lie group. Then the inclusion $H\le G$ satisfies the index condition: for every $x\in G$, the equality $[H: H\cap xHx^{-1}]=[H: H\cap x^{-1}Hx]$ holds.
\end{proposition}
\begin{proof}
In this case, we can apply Lemma \ref{Lie case} directly to $K=H\cap xHx^{-1}\le H$ and $K=H\cap x^{-1}Hx \le H$ as subgroups of $G.$
\end{proof}

Note that a left coset and a right coset of $H$ in $G$ have a common representative if and only if they lie in the same double coset. 
Indeed, we recall the following well-known lemma.

\begin{lemma}
Let $x,x' \in G$. The following conditions are equivalent:
\begin{enumerate}
    \item The cosets $xH$ and $Hx'$ intersect.
    \item The cosets $xH$ and $Hx'$ have a common representative, i.e.\ there exists $x'' \in G$ such that $xH=x''H$ and $Hx'=Hx''$.
    \item The cosets $xH$ and $Hx'$ lie in one double coset, i.e. there exists $x'' \in H$ such that $xH \cup Hx' \subset Hx''H$.
 \end{enumerate}
\end{lemma}
\begin{proof}
$(1) \Rightarrow (2):$ Take $x'' \in xH \cap Hx'$. $(2) \Rightarrow (3):$ The $x''$ that worked for $(2)$ also works for $(3)$. $(3) \Rightarrow (1):$ There exists $h_1,h_2,h'_1,h'_2$ such that $h_1x''h_2=x$ and $h_1'x''h_2'=x'$. This implies that $h_1^{-1} x h_2^{-1}= x'' =h_1'^{-1} x' h_2'^{-1}$ and hence
$x h_2^{-1}h_2' = h_1h_1'^{-1} x'$. Thus, $xH$ and $Hx'$ intersect.
\end{proof}

Let us illustrate the use of the index condition by constructing a common Borel transversal for a single double coset. This is done in the following lemma. The index condition ensures that the number of left and right cosets contained in one double coset coincides. The proof of the main theorem will deal with the problem of doing the same construction in a Borel way for all double cosets at the same time. The argument below modulo the measurability issue is essentially contained in Ore's work \cite{MR0100639Ore1958}. 

\begin{lemma} Let $G$ be a Polish group and $H$ be a compact subgroup satisfying the index condition. For each $x \in G$, there exists a Borel set $Q \subset HxH$ which is simultaneously a set of left and right coset representatives with respect to $H$, i.e.
$$HxH = \sqcup_{a \in Q} Ha = \sqcup_{a \in Q} aH.$$
\end{lemma}
\begin{proof}
Let $x \in G$ be fixed and consider the double coset $HxH \subset G.$ Note that each left coset $t'H$ contained in $HxH$ is of the form $txH$ for some $t \in H.$ Similarly, right cosets contained in $Hx$ are of the form $Hxs$ for $s \in H.$
    
Two cosets $txH$ and $xH$ coincide with $t \in H$ if and only if $t \in H \cap xHx^{-1}$. 
Let $T \subset H$ be a Borel section of the quotient map $H \to H \slash (H \cap xHx^{-1})$. 
This yields a decomposition $H= \sqcup_{t \in T} t(H \cap xHx^{-1}).$
It follows that $T \subset H$ also yields a decomposition of $HxH$ in the sense that we have $HxH = \sqcup_{t \in T} txH.$

Similarly, two cosets $Hxs$ and $Hx$ coincide if and only if $s \in H \cap x^{-1}Hx$. We take $S$ a Borel section of the map $H \to (H \cap x^{-1}Hx)\backslash H$ and observe that
$HxH = \sqcup_{s \in S} Hxs.$

Note that both $S$ and $T$ are standard Borel spaces, i.e.\ they are either finite or Borel isomorphic to $[0,1]$.
By Proposition \ref{prop same index}, $[H\colon H\cap xHx^{-1}]=[H\colon H\cap x^{-1}Hx]\in \N_{\geq 1}\cup \{2^{\omega}\}$ and hence the cardinalities of $S$ and $T$ agree. Thus, we may pick a Borel isomorphism $\varphi \colon T \to S$
and set $Q := \{ t x \varphi(t) \mid t \in T\} \subset HxH.$ It is straightforward to check that $Q$ is a Borel subset of $HxH$ which is simultaneously a set of left and right coset representatives, i.e., we have:
$$HxH = \sqcup_{a \in Q} Ha = \sqcup_{a \in Q} aH.$$
This finishes the proof.
\end{proof}

\begin{remark}
Sometimes, the index condition is automatically true provided both indices are known to be finite. Consider for example $H \le G$, when $H$ is a free group. Then, the Nielsen-Schreier formula implies that
$$[H:K] = \frac{{\rm rk}(K)-1}{{\rm rk}(H)-1},$$ provided the index is finite. In particular, the index does only depends on $K$ up to isomorphism and does not change, when $K$ is replaced by an isomorphic finite index subgroup of $H$. A similar phenomenon occurs, when the group $H$ has a finite and non-zero $\ell^2$-Betti number, since 
$$[H:K] = \frac{\beta_k^{(2)}(K)}{\beta_k^{(2)}(H)}.$$

This applies for example to the inclusion ${\rm SL}(2,\mathbb Z) \le {\rm SL}(2,\mathbb Q).$ Since
${\rm SL}(2,\mathbb Q)$ commensurates ${\rm SL}(2,\mathbb Z)$, all inclusions ${\rm SL}(2,\mathbb Z) \cap g{\rm SL}(2,\mathbb Z)g^{-1} \le {\rm SL}(2,\mathbb Z)$ for $g \in {\rm SL}(2,\mathbb Q)$ are of finite index. Since $\beta_1^{(2)}({\rm SL}(2,\mathbb Z)) =\frac1{12},$ the index condition is satisfied by the previous observation.

Note however, that this observation does not in general prevent the existence of isomorphic subgroup of infinite index.
\end{remark}

\begin{remark}
The index condition is satisfied if $G$ is a locally compact, unimodular, totally disconnected group and $H \le G$ is a compact open subgroup. Indeed, 
$$[H : H \cap gHg^{-1}] = \frac{\mu(H)}{\mu(H \cap gHg^{-1})}$$
in this case. Now, since $G$ is unimodular, the measure of $H \cap gHg^{-1}$ agrees with the measure of $H \cap g^{-1}Hg$, since these groups are conjugate. In particular, this applies to ${\rm SL}(n,\mathbb Z_p) \le {\rm SL}_n(\mathbb Q_p).$
\end{remark}

\section{Proof of Theorem \ref{thm main}}

We are now heading towards the proof of Theorem \ref{thm main}. As a first step, we need to find a Borel set of representatives of double cosets of $H$ in $G$.

\begin{lemma}\label{lem borel transverssal for doublecoset}
Let $G$ be a Polish group, $H$ be a compact subgroup of $G$. Then there exists a Borel subset $A\subset G$ which is a set of representatives for the double $H$-cosets: $G=\bigsqcup_{a\in A}HaH$.
\end{lemma}

\begin{proof}
Let $\pi\colon G\to H\setminus G/H$ be the quotient map. 
Since $G$ is a Polish group and $H$ is comact, the space $H\setminus G/H$ is also Polish. 
For each $g\in G$, we may view $\pi(g)=HgH$ as a closed subset of $G$, whence we may view $H\setminus G/H\subset \mathcal{F}^*(G)$. Note that the inclusion map $j\colon H\setminus G/H\to \mathcal{F}^*(G)$ is Borel. To see this, recall that $A\subset H\setminus G/H$ is Borel if and only if $\pi^{-1}(A)$ is a Borel subset of $G$. This follows from the fact that $\tilde{\mathcal{B}}=\{A\subset H\setminus G/H\mid \pi^{-1}(A)\text{\ is Borel\ }\}$ is a $\sigma$-algebra containing all open subsets of $H\setminus G/H$, thus containing all of its Borel subsets. Then if $U$ is any open subset of $G$ and $V=\{F\in \mathcal{F}(G)\mid F\cap U\neq \emptyset\}$, then 
$\pi^{-1}(j^{-1}(A))=\{g\in G\mid HgH\cap U\neq \emptyset\}=HUH$ is open in $G$, whence  it is Borel. The above remark then shows that $j^{-1}(V)$ is Borel. Thus $j$ is a Borel map. By Theorem \ref{thm Borel selector}, there exists a Borel map $\sigma \colon \mathcal{F}^*(G)\to G$ such that $\sigma (F)\in F$ for every $F\in \mathcal{F}^*(G)$. Define $s=\sigma\circ j\circ \pi\colon G\to G$ and $A=\{g\in G\mid s(g)=g\}$, which is a Borel subset of $G$. Note that $s(g)=\sigma(HgH)$ for every $g\in G$, and because $\sigma(HgH)\in HgH$, we have $\pi(s(g))=\pi(\sigma(HgH))=HgH$ and $s(s(g))=s(g)$. In particular, $s(g)\in A$ for every $g\in G$. Thus $G=\bigcup_{a\in A}HaH$ holds. If $a,a'\in A$, satisfy $HaH\cap Ha'H\neq \emptyset$, then $\pi(a)=HaH=Ha'H=\pi(a')$ and thus 
\eqa{
    a'&=s(a')=\sigma\circ j\circ \pi(a')\\
    &=\sigma \circ j\circ \pi(a)\\
    &=s(a)=a. 
}
Thus $G=\bigsqcup_{a\in A}HaH$ is the decomposition for the double $H$-cosets. 
\end{proof}

\begin{lemma}
    Let $X, Y$ be Polish spaces, $f\colon X\to Y$ be a continuous open surjection. Then the map $\iota_f\colon Y\ni y\mapsto f^{-1}(\{y\})\in \mathcal{F}^*(X)$ is Borel. 
\end{lemma}
\begin{proof}
Since $f$ is a continuous surjection, $f^{-1}(\{y\})\in \mathcal{F}^*(X)$ holds for every $y\in Y$.\\ 
It suffices to show that for each open set $U$ in $X$, the set $\iota_f^{-1}(B)$ is Borel in $Y$, where $B=\{F\in \mathcal{F}^*(X)\mid F\cap U\neq \emptyset\}$.
But \eqa{
\iota_f^{-1}(B)&=\{y\in Y\mid f^{-1}(\{y\})\cap U\neq \emptyset\}\\
&=f(U),
}
which is open becuase $f$ is an open mapping. Therefore $\iota_f$ is Borel.
\end{proof}
\begin{remark} It is unclear if $\iota_f$ can be continuous for an open continuous surjection $f\colon X\to Y$. Note that if $f$ is not assumed to be open, then $\iota_f$ is often discontinuous.
For example, 
take $X=[0,1]\cup [2,3]$ and $Y=[0,2]$. Define a continuous function $f\colon X\to Y$ by $f(t)=t$ for $t\in [0,1]$ and $f(t)=t-1$ for $t\in [2,3]$. 
Let $y_n=1-\tfrac{1}{n}\,(n\in \N)$ and $y=1$. Then $y_n\xrightarrow{n\to \infty}y$, but $\iota_f(y_n)=\{1-\tfrac{1}{n}\}$, $\iota_f(y)=\{1,2\}$. Since the set $\mathcal{F}^*_{\le 1}(X)$ is closed, $\iota_f(y_n)\not\xrightarrow{n\to \infty}\iota_f(y)$. Actually $\iota_f(y_n)\xrightarrow{n\to \infty}\{1\}\subsetneq \iota_f(y)$. Thus $\iota_f(Y)$ is in general not closed in $\mathcal{F}^*(X)$, and the set $\{y\in Y\mid \sharp f^{-1}(\{y\})\le n\}$ is not closed either. 

\end{remark}
\begin{lemma}\label{lem: F_len is closed}
    Let $X$ be a Polish space, and $n\in \N$. Then with respect to the Wijsman topology, the set $\mathcal{F}^*_{\le n}(X)=\{F\in \mathcal{F}^*(X)\mid \sharp F\le n\}$ is closed. In particular, $\mathcal{F}^*_n(X)=\{F\in \mathcal{F}^*(X)\mid \sharp F=n\}$ is Borel. 
\end{lemma}
\begin{proof}
Let $(F_k)_{k=1}^{\infty}$ be a sequence in $\mathcal{F}_{\le n}^*(X)$ converging to $F\in \mathcal{F}^*(X)$. Write $F_k=\{x_1^{(k)},\dots,x_n^{(k)}
\}$ (we do not assume that $x_i^{(k)}$'s are different for different $i$). 
Assume by contradiction that there exist $n+1$ distinct points $x_1,\dots, x_{n+1}\in F$. By definition we have $d(x_i,F_k)\xrightarrow{k\to \infty}d(x_i,F)=0$. Fix $i\in \{1,\dots,n+1\}$. For each $k\in \N$, there exists $j(k,i)\in \{1,\dots,n\}$ such that $d(x_i,x_{j(k,i)}^{(k)})=d(x_i,F_k)$. Thus there exists $j(i)\in \{1,\dots, n\}$ such that $j(k,i)=j(i)$ for infinitely many $k$. By passing to a subsequence up to $n+1$ times, we may assume that $j(k,i)=j(i)$ for each $i\in \{1,\dots,n+1\}$ and $k\in \N$. 
Thus $\displaystyle \lim_{k\to \infty}d(x_i,x_{j(i)}^{(k)})=0$ for every $i\in \{1,\dots,n+1\}$. There exist at least two elements $i,i'\in \{1,\dots, n+1\}, i<i'$, such that $j(i)=j(i')$. Then the sequence $(x_{j(i)}^{(k)})_{k=1}^{\infty}$ converges to two points $x_i,x_{i'}$, a contradiction. Therefore $\sharp F\le n$. 
\end{proof}
\begin{remark}
The set $\mathcal{F}^*_n(X)=\{F\in \mathcal{F}^*(X)\mid \sharp F=n\}$ is not closed. For example, consider $X=[-1,1]$ with the Euclidean metric $d$, and set $F_n=\{\pm \tfrac{1}{n}\} (n\in \N),\,F=\{0\}$. Then for each $t\in [-1,1]$, we have 
$d(t,F_n)=\min \{|t-\tfrac{1}{n}|,|t+\tfrac{1}{n}|\}\xrightarrow{n\to \infty}|t|=d(t,F)$. Thus $F_n\xrightarrow{n\to \infty}F$, $\sharp F_n=2\,(n\in \N)$ and $\sharp F=1$. 
\end{remark}

\begin{proposition}\label{prop finite fibers Fsigma}
Let $H$ be a compact metrizable group acting continuously on a Polish space $X$. Let $f\colon X\to Y=H\backslash X$ be the quotient map. 
Consider the map $\iota_f\colon Y\ni y\mapsto f^{-1}(\{y\})\in \mathcal{F}^*(X)$. 
Then the following statements hold. 
\begin{list}{}{}
\item[{\rm (i)}] For each $y\in Y$, $\iota_f(y)$ is either finite or perfect.
\item[{\rm (ii)}] $\iota_f$ is continuous. Consequently, the set $\{y\in Y\mid \sharp \iota_f(y)\le n\}$ is closed for every $n\in \N$. In particular, each $Y_n=\{y\in Y\mid \sharp \iota_f(y)=n\}$ is Borel and the subspace $\{y\in Y\mid \sharp \iota_f(y)=\infty\}$ is $G_{\delta}$, whence Polish.  
\end{list}
\end{proposition}
We will use the following well-known result. Suppose $G$ is a compact metrizable group acting on a Polish space $X$. Let ${\rm d}g$ be the normalized Haar measure on $G$.  
Let $d_0$ be a compatible complete metric on $X$ with diameter 1. Define $d\colon X\times X\to [0,1]$ by 
\[d(x,y):=\int_Gd_0(gx,gy)\,{\rm d}g,\ \ \ \ x,y\in X.\]
Then by standard arguments, one can show that $d$ is a compatible complete metric on $X$ which is $G$-invariant. We summarize this observation. 
\begin{lemma}\label{lem: invmetric}
Let $\alpha\colon G\curvearrowright X$ be a continuous action of a compact Polish group $G$ on a Polish space $X$. 
Then there exists a complete metric $d$ on $X$ compatible with the topology which is $G$-invariant, i.e., $d(gx,gy)=d(x,y)$ holds for every $x,y\in X$ and $g\in G$. 
\end{lemma}

\begin{proof}[Proof of Proposition \ref{prop finite fibers Fsigma}]
(i) Take $y\in Y$. $\iota_f(y)$ is compact hence a Polish space. Assume first that $\iota_f(y)$ is countable. 
    Then by the Baire category theorem, there is an isolated point $x$ in $\iota_f(y)$. On the other hand, $H$ acts transitively on $\iota_f(y)$, which implies that each point in $\iota_f(y)$ is isolated, i.e., $\iota_f(y)$ is discrete. Since $\iota_f(y)$ is compact, this shows that $\iota_f(y)$ is finite. 
    Assume next that $\iota_f(y)$ is uncountable. If there is an isolated point in $\iota_f(y)$, then by the same argument as before $\iota_f(y)$ is discrete, which contradicts the separability of the Polish space $\iota_f(y)$. Thus $\iota_f(y)$ is perfect.\\
(ii)
Let $d$ be a metric on $X$ compatible with the topology which is $H$-invariant (Lemma \ref{lem: invmetric}).
Let $(y_n)_{n=1}^{\infty}$ be a sequence in $Y$ converging to $y\in Y$. 
Then there exists a sequence $(x_n)_{n=1}^{\infty}$ in $X$ converging to $x\in X$ such that $f(x_n)=y_n\,(n\in \N)$ and $f(x)=y$. Indeed, take any $x\in X$ such that $f(x)=y$. Since $X$ is metrizable, there exists a neighborhood basis $\{U_k\}_{k=1}^{\infty}$ of $x$ satisfying $U_k\supset U_{k+1}\,(k\in \N)$. 
Since $f(U_k)$ is an open neighborhood of $y$, we may find an increasing sequence $N_1<N_2<\dots$ of natural numbers such that 
$y_n\in f(U_k)$ for every $n\ge N_k$. Therefore for each $n\in \N$, there exists $x_n\in X$ such that $f(x_n)=y_n$ and $x_n\in U_k$ hold for every $n\ge N_k$. It is then clear that $\disp \lim_{n\to \infty}x_n=x$. 
We show that $\disp \lim_{n\to \infty}\iota_f(y_n)=\iota_f(y)$ in the Wijsman topology. Since $\iota_f(y_n)=Hx_n$ and $\iota_f(y)=Hx$, this amounts to show that $\disp \lim_{n\to \infty}d(x',Hx_n)=d(x',Hx)$ for every $x'\in X$. 
Fix $x'\in X$. Since $d$ is $H$-invariant, we have 
\begin{eqnarray*}
d(x',Hx_n)&=&\min_{h\in H}d(x',hx_n) \\ &=&\min_{h\in H}d(h^{-1}x',x_n)=d(x_n,Hx')\xrightarrow{n\to \infty}d(x,Hx')=d(x',Hx),
\end{eqnarray*}
where we used the fact that the map $d(\cdot,Hx')$ is continuous on $X$. This shows that $\iota_f$ is continuous. 
Then for each $n\in \N$, the set $\{y\in Y\mid \sharp \iota_f(y)\le n\}$ is the inverse image of the set $\mathcal{F}^*_{\le n}(X)$ (which is closed by Lemma \ref{lem: F_len is closed}) by the continuous map $\iota_f$, whence it is also closed. 
Therefore $\{y\in Y\mid \sharp \iota_f(y)<\infty\}=\bigcup_{n\in\N}\{y\in Y\mid \sharp \iota_f(y)\le n\}$ is $F_{\sigma}$ and its complement $\{y\in Y\mid \sharp \iota_f(y)=\infty\}$ is a $G_{\delta}$ subset of $Y$, whence Polish. 
\end{proof}
We will use Mauldin's Borel parametrization theorem \cite[Theorem A]{MauldinMR530052}.
\begin{definition}
Let $X,Y$ be Polish spaces, $B$ be a Borel subset of $X\times Y$. 
A {\it Borel parametrization} of $B$ is a Borel isomorphism $g\colon X\times E\to B$ such that $g(x,\cdot)$ is a Borel isomorphism of $E$ onto $B_x$, where $E$ is a Borel subset of $Y$. 
\end{definition}
If all $B_x$ are uncountable, then because any uncountable standard Borel spaces are Borel isomorphic, we may replace $E$ by $2^{\omega}$ in the definition of the Borel parametrization. 
\begin{theorem}[Mauldin \cite{MauldinMR530052}]\label{thm Mauldin}
Let $X,Y$ be Polish spaces, $B$ be a Borel subset of $X\times Y$ such that $B_x$ is uncountable for every $x\in X$. 
Then $B$ admits a Borel parametrization if and only if there exists a Borel subset $M$ of $B$ such that for every $x\in X$, $M_x$ is a nonempty compact perfect set.
\end{theorem}
\begin{proposition}\label{prop borel parametrization for compact}
Let $H$ be a compact metrizable group acting continuously on a Polish space $X$. Denote by $f\colon X\to Y=H\backslash X$ the quotient map. 
Define Borel sets $X_n=f^{-1}(Y_n), X_{\infty}=f^{-1}(Y_{\infty})$, where $Y_n=\{y\in Y\mid \sharp \iota_f(y)=n\}\,(n\in \N)$ and $Y_{\infty}=\{y\in Y\mid \sharp \iota_f(y)=\infty\}$. 
Then there exist Borel isomorphisms $g_n\colon Y_n\times \{1,\dots,n\}\to X_n\,(n\in \N)$ and $g_{\infty}\colon Y_{\infty}\times 2^{\omega}\to X_{\infty}$ such that 
for each $n\in \N$ and $y\in Y_n$, $g_n(y,\cdot)$ is a Borel isomoprhism of $\{1,\dots,n\}$ onto $f^{-1}(\{y\})$, and for each $y\in Y_{\infty}$, $g_{\infty}(y,\cdot)$ is a Borel isomorphism of $2^{\omega}$ onto $f^{-1}(\{y\})$. 
\end{proposition}
\begin{proof}
By Proposition \ref{prop finite fibers Fsigma}, the sets $X_n,Y_n\,(n\in \N)$ are Borel, $X_{\infty}$ (resp. $Y_{\infty})$ are $G_{\delta}$ hence a Polish subspace of $X$ (resp. $Y$). We may assume that all $X_n,Y_n,X_{\infty},Y_{\infty}$ are all nonempty. 
First, we construct a Borel isomoprhism $g_n\colon Y_n\times \{1,\dots,n\}\to X_n$. For $n=1$, $g_1(y,1)$ is the unique point in $f^{-1}(\{y\})$ and it is Borel because the map $f|_{X_1}$ is a Borel isomorphism with inverse $g_1(\cdot,1)$. Assume $n\ge 2$. Choose a Polish topology on $Y_n$ whose Borel structure coincides with the subspace Borel structure on $Y_n$. By \cite[Theorem 13.11]{kechrisMR1321597}, there exists a Polish topology on $X_n$ with the same Borel structure as its subspace Borel structure such that $f_n=f|_{X_n}\colon X_n\to Y_n$ is continuous. By Theorem \ref{thm Borel selector}, there exists a Borel map $\sigma_n\colon \mathcal{F}^*(X_n)\to X_n$ such that $\sigma_n(F)\in F$ for every $F\in \mathcal{F}^*(X_n)$. Since $f_n$ is continuous and countable-to-1, $\iota_{f_n}$ is Borel. Indeed, this is a standard corollary of the Lusin--Novikov uniformization theorem, see \cite[Theorem 18.10]{kechrisMR1321597}).

Note that the composition $\sigma_n\circ \iota_{f_n}\colon Y_n\to X_n$ is an injective Borel map. Indeed if $y,y'\in Y_n$ satisfies $\sigma_n(\iota_{f_n}(y))=\sigma_{n}(\iota_{f_n}(y'))$, then this element belongs to $f^{-1}(\{y\})\cap f^{-1}(\{y'\})$, which implies that $y=y'$. Therefore $\sigma_n(\iota_{f_n}(Y_n))$ is Borel in $X_n$. 
Then the map $Y_n\ni y\mapsto \sigma_n(\iota_{f_n}(y))\in X_{n,n}=\sigma_n(\iota_{f_n}(Y_n))$ is a Borel isomorphism, which we call $g_n(\cdot,n)$.  
Set $X_n'=X_n\setminus X_{n,n}$ and $f_n'=f_n|_{X_n'}\colon X_n'\to Y_n$. Then $f_n'$ is a Borel surjection with $\sharp \iota_{f_n'}(y)=n-1$ for every $y\in Y_n$. 
By applying the same argument repeatedly $n-1$ times, we obtain a Borel partition $X_{n,1}\sqcup X_{n,2}\sqcup \cdots \sqcup X_{n,n-1}=X_n'$ and Borel isomorphisms $g_n(\cdot,k)\colon Y_n\to X_{n,k}\,(k=1,\dots,n-1)$. Thus we may define a Borel isomorphism $g_n\colon Y_n\times \{1,\dots,n\}\to X_n$ such that $g_n(y,\cdot)\colon \{1,\dots,n\}$ is a Borel isomorphism of $\{1,\dots, n\}$ onto $f^{-1}(\{y\})$.\\
Next, We view $X_{\infty}$ and $Y_{\infty}$ as Polish in their subspace topologies. Let $B_{\infty}=\{(f(x),x)\mid x\in X_{\infty}\}$, which is the image of the injective Borel map $X_{\infty}\ni x\mapsto (f(x),x)\in Y_{\infty}\times X_{\infty}$, whence a Borel subset of $Y_{\infty}\times X_{\infty}$. Let $y\in Y_{\infty}$. By Proposition \ref{prop finite fibers Fsigma} (i), $(B_{\infty})_y=f^{-1}(\{y\})\subset X_{\infty}$ is perfect and compact in $X$, whence it is perfect and compact in $X_{\infty}$ as well. By Theorem \ref{thm Mauldin}, there exists a Borel isomorphism $g_{\infty}\colon Y_{\infty}\times 2^{\omega}\to B_{\infty}$ such that $g_{\infty}(y,\cdot)$ is a Borel isomoprhism of $2^{\omega}$ onto $(B_{\infty})_y=f^{-1}(\{y\})$ for every $y\in Y_{\infty}$.   
\end{proof}

Now we are ready to prove that a common Borel transversal for $H\le G$ exists. 

\begin{proof}[Proof of Theorem \ref{thm main}]
 By Lemma \ref{lem borel transverssal for doublecoset}, there exists A Borel subset $A$ of $G$ which is a set of representatives for the double $H$-cosets. 
Thus the restriction of the quotient map $\pi\colon G\ni x\mapsto HxH\in H\backslash G/H$ to $A$ is a Borel isomorphism. Let $X=G/H, Y=H\backslash G$ and $Z=H\backslash G/H$, which are Polish because so is $G$ and $H$ is compact. Let $f\colon X\to Z$ and $g\colon Y\to Z$ be the quotient maps. 
{By the index condition for $H\le G$, we have $\sharp f^{-1}(\{z\})=\sharp g^{-1}(\{z\})$ for each $z\in Z$. Thus as in Proposition \ref{prop borel parametrization for compact}, We define $Z_n=\{z\in Z\mid \sharp \iota_f(z)=n\}, X_n=f^{-1}(Z_n), Y_n=g^{-1}(Z_n)$ for each $n\in \N$. We also define $Z_{\infty}=\{z\in Z\mid \sharp \iota_f(z)=\infty\},\,X_{\infty}=f^{-1}(Z_{\infty})$ and $Y_{\infty}=g^{-1}(Z_{\infty})$.}  
Let $x\in A$. 
We also define $A_n=A\cap \pi^{-1}(Z_n)\,(n\in \N)$ and $A_{\infty}=A\cap \pi^{-1}(Z_{\infty})$. All of these sets are Borel.  
First, we consider the case $x\in A_{\infty}$. By Proposition \ref{prop borel parametrization for compact}, there exists a Borel isomorphism $\varphi_{\infty}\colon Z_{\infty}\times 2^{\omega}\to X_{\infty}$ 
(resp. $\psi_{\infty}\colon Z_{\infty}\times 2^{\omega}\to Y_{\infty}$) such that $\varphi_{\infty}(z,\cdot)\colon 2^{\omega}\to f^{-1}(\{z\})$ (resp. $\psi_{\infty}(z,\cdot)\colon 2^{\omega}\to g^{-1}(\{z\})$) is a Borel isomorphism for every $z\in Z_{\infty}$. 
Choose $\tilde{h}(x,\alpha)\in H$ (resp. $\tilde{k}(x,\alpha)\in H)$ such that $\varphi_{\infty}(\pi(x),\alpha)=\tilde{h}(x,\alpha)xH$ (resp. $\psi_{\infty}(\pi(x),\alpha)=Hx\tilde{k}(x,\alpha)$). Observe that because there are $\sharp (H\cap xHx^{-1})$ (resp. $\sharp (H\cap x^{-1}Hx$)) many choices for such $\tilde{h}(x,\alpha)$ (resp. $\tilde{k}(x,\alpha)$), it is unclear whether the maps $\tilde{h},\tilde{k}\colon A_{\infty}\times 2^{\omega}\to H$ are Borel. We resolve this issue as follows. For each $x\in A_{\infty}$, let $\hat{h}(x,\alpha)$ be the image of $\tilde{h}(x,\alpha)$ in the quotient space $H/(H\cap xHx^{-1})$. 
Then $\hat{h}(x,\alpha)$ is independent of the choice of $\tilde{h}(x,\alpha)$. Since each point in $H/(H\cap xHx^{-1})$ is a closed subset of $H$, we view 
\[\hat{h}(x,\alpha)=\tilde{h}(x,\alpha)(H\cap xHx^{-1})\in \mathcal{F}^*(H).\]
Then we show that the map $\hat{h}\colon A_{\infty}\times 2^{\omega}\to \mathcal{F}^*(H)$ is Borel. Let $U$ be a nonempty open subset of $H$ and let $\tilde{U}=\{F\in \mathcal{F}^*(H)\mid F\cap U\neq \emptyset\}$. 
Then $\hat{h}^{-1}(\tilde{U})=\{(x,\alpha)\in A_{\infty}\times 2^{\omega}\mid \hat{h}(x,\alpha)\cap U\neq \emptyset\}$, and 
\eqa{
    \hat{h}(x,\alpha)\cap U\neq \emptyset & \iff \tilde{h}(x,\alpha)(H\cap xHx^{-1})\cap U\neq \emptyset\\
    &\iff \tilde{h}(x,\alpha)\in U(H\cap xHx^{-1})\\
    &\stackrel{(*)}{\iff}\varphi_{\infty}(\pi(x),\alpha)\in U[xH]:=\{uxH\mid u\in U\}\subset \mathcal{F}^*(G).
}
To see $(*)$, suppose $\tilde{h}(x,\alpha)\in U(H\cap xHx^{-1})$. Then 
$\tilde{h}(x,\alpha)=u(x,\alpha)s(x,\alpha)$ 
for some $u(x,\alpha)\in U$ and $s(x,\alpha)\in H\cap xHx^{-1}$. Thus 
$\varphi_{\infty}(\pi(x),\alpha)=\tilde{h}(x,\alpha)xH=u(x,\alpha)s(x,\alpha)xH=u(x,\alpha)xH\in U[xH]$
 because 
 \[(u(x,\alpha)s(x,\alpha))^{-1}u(x,\alpha)=s(x,\alpha)^{-1}\in H\cap xHx^{-1}.\]
 Conversely, if $\varphi_{\infty}(\pi(x),\alpha)\in U[xH]$, then there exists $u(x,\alpha)\in U$ such that $\tilde{h}(x,\alpha)xH=u(x,\alpha)xH$, which implies that 
 $u(x,\alpha)^{-1}\tilde{h}(x,\alpha)\in H\cap xHx^{-1}$, or equivalently 
 $\tilde{h}(x,\alpha)\in U(H\cap xHx^{-1})$. 
 Thus 
$\hat{h}^{-1}(\tilde{U})=\{(x,\alpha)\in A_{\infty}\times 2^{\omega}\mid \varphi_{\infty}(\pi(x),\alpha)\in U[xH]\}$, and we are going to show that it is Borel. 

{Fix a complete metric $d$ on $G$ compatible with the topology, with respect to which we consider the Hausdorff metric $\delta$ on the space $\mathcal{K}^*(G)\subset \mathcal{F}^*(G)$. 
By Lemma \ref{lem Wijsman=Vietoris}, the metric topology of $\delta$ is compatible with the restriction of the Effros Borel structure on $\mathcal{F}^*(G)$ to $\mathcal{K}^*(G)$}.

Note that $\hat{h}^{-1}(\tilde{U})$ coincides with 
\[\widehat{U}:=\{(x,\alpha)\in A_{\infty}\times 2^{\omega}\mid H\in x^{-1}U^{-1}[\varphi_{\infty}(\pi(x),\alpha)]\},\]
where $x^{-1}U^{-1}[\varphi_{\infty}(\pi(x),\alpha)]:=\{x^{-1}u^{-1}\varphi_{\infty}(\pi(x),\alpha)\}\subset \mathcal{F}^*(G)$. 
Since $U$ is a nonempty open set in a metrizable space $H$, it is $F_{\sigma}$, thus $U=\bigcup_{i=1}^{\infty}F_i$ for some $F_1,F_2,\dots\in \mathcal{F}^*(H)$. Then $\widehat{U}=\bigcup_{i=1}^{\infty}\widehat{F_i}$, 
whence it suffices to show that $\hat{h}^{-1}(F)$ is Borel for every closed set $F$. Fix $F\in \mathcal{F}^*(H)$. Then $F$ is separable, so we may choose a countable dense subset $\{k_n\mid n\in \N\}$ of $F$. 
To show that $\hat{h}^{-1}(F)$ is Borel, we first observe that
\[H\in x^{-1}F^{-1}[\varphi_{\infty}(\pi(x),\alpha)]\iff \inf_{n\in \N}\delta(H,x^{-1}k_n^{-1}\varphi_{\infty}(\pi(x),\alpha))=0.\]

{
Indeed, if $H=x^{-1}k^{-1}\varphi_{\infty}(\pi(x),\alpha)$ for some $k\in F$, there exists a sequence $(k_{n_i})_{i=1}^{\infty}$ such that $d(k_{n_i},k)\xrightarrow{i\to \infty}0.$ Then $\delta(H,x^{-1}k_{n_i}^{-1}\varphi_{\infty}(\pi(x),\alpha))$ is equal to the maximum of
\eqa{
    \max_{g\in \varphi_{\infty}(\pi(x),\alpha)}d(x^{-1}k^{-1}g,x^{-1}k_{n_i}^{-1}\varphi_{\infty}(\pi(x),\alpha))
}
and
\eqa{
    \max_{g\in \varphi_{\infty}(\pi(x),\alpha)}d(x^{-1}k_{n_i}^{-1}g,x^{-1}k^{-1}\varphi_{\infty}(\pi(x),\alpha)).
}
The first part can be estimated from above by
\eqa{
    \max_{g\in \varphi_{\infty}(\pi(x),\alpha)}
    d(x^{-1}k^{-1}g,x^{-1}k_{n_i}^{-1}g)
    \xrightarrow{i\to \infty}0
}
 because $\varphi_{\infty}(\pi(x),\alpha)$ is compact, whence $x^{-1}k_{n_i}^{-1}g$ converges to $x^{-1}k^{-1}g$ uniformly on $g\in \varphi_{\infty}(\pi(x),\alpha)$.
}

Likewise, 
\[\lim_{i\to \infty}\max_{g\in \varphi_{\infty}(\pi(x),\alpha)}d(x^{-1}k_{n_i}^{-1}g,x^{-1}k^{-1}\varphi_{\infty}(\pi(x),\alpha))=0.\]
Therefore $\disp \inf_{n\in \N}\delta(H,x^{-1}k_n^{-1}\varphi_{\infty}(\pi(x),\alpha))=0$ holds.\\
Conversely, assume $\inf_{n\in \N}\delta(H,x^{-1}k_n^{-1}\varphi_{\infty}(\pi(x),\alpha))=0$. Then there exists $(k_{n_i})_{i=1}^{\infty}$ such that $\disp \lim_{i\to \infty}\delta(H,x^{-1}k_{n_i}^{-1}\varphi_{\infty}(\pi(x),\alpha))=0$. 
Since $F$ is compact, by passing to a subsequence we may assume that the limit $\disp k=\lim_{i\to \infty}k_{n_i}\in F$ exists. 
Then by the same argument as above, the sequence $$(x^{-1}k_{n_i}^{-1}\varphi_{\infty}(\pi(x),\alpha))_{i=1}^{\infty}$$ in {$\mathcal{K}^*(G)$} is $\delta$-convergent to both $H$ and $x^{-1}k^{-1}\varphi_{\infty}(\pi(x),\alpha)$, whence by the Hausdorff property, $H=x^{-1}k^{-1}\varphi_{\infty}(\pi(x),\alpha)\in x^{-1}F^{-1}[\varphi_{\infty}(\pi(x),\alpha)]$ holds. Since {$\varphi_{\infty}(\pi(\cdot),\cdot)\colon A_{\infty}\times 2^{\omega}\to X_{\infty}\subset \mathcal{K}^*(G)\subset \mathcal{F}^*(G)$} is Borel and {$\delta(H,\cdot)\colon \mathcal{K}^*(G)\to \R$} is Borel, it follows that $\hat{h}^{-1}(F)$ is Borel. 

This shows that $\hat{h}\colon A_{\infty}\times 2^{\omega}\to \mathcal{F}^*(H)$ is Borel. 
By Theorem \ref{thm Borel selector}, there exists a Borel map $\sigma_H\colon \mathcal{F}^*(H)\to H$ such that $\sigma_H(F)\in F$ for every $F\in \mathcal{F}^*(H)$. Then define $h=\sigma_H\circ \hat{h}\colon A_{\infty}\times 2^{\omega}\to H$, which is Borel. 
Then by construction, for each $(x,\alpha)\in A_{\infty}\times 2^{\omega}$, the class of $h(x,\alpha)$ in $H/(H\cap xHx^{-1})$ is $\hat{h}(x,\alpha)$, which is also the class of $\tilde{h}(x,\alpha)$ by definition. Therefore it follows that 
\[h(x,\alpha)xH=\tilde{h}(x,\alpha)xH=\varphi_{\infty}(\pi(x),\alpha).\]
By the same argument, we may find a Borel map $k\colon A_{\infty}\times 2^{\omega}\to H$ such that 
\[xHk(x,\alpha)=\psi_{\infty}(\pi(x),\alpha).\]
Then $T_x=\{h(x,\alpha)\mid \alpha\in 2^{\omega}\}$ (resp. $S_x=\{k(x,\alpha)\mid \alpha\in 2^{\omega}\}$) is a transversal for 
$H\to H/(H\cap xHx^{-1})$ (resp. $H\to (H\cap x^{-1}Hx)\backslash H$). Thus the map $A_{\infty}\times 2^{\omega}\ni (x,\alpha)\mapsto h(x,\alpha)xk(x,\alpha)\in G$ is an injective Borel map, whence $T_{\infty}=\{h(x,\alpha)xk(x,\alpha)\mid \alpha\in 2^{\omega},x\in A_{\infty}\}$ is a Borel subset of $G$. Moreover, it is straightforward to see that 
\[\pi^{-1}(Z_{\infty})=\bigsqcup_{g\in T_{\infty}}Hg=\bigsqcup_{g\in T_{\infty}}gH.\]
Next, we consider the case $x\in A_n$. 
Since the arugments are essentially identical to the $A_{\infty}$ case, we describe the argument briefly. By Proposition \ref{prop borel parametrization for compact}, there exists a Borel isomorphism $\varphi_n\colon Z_n\times \{1,\dots,n\}\to X_n$ (resp. $\psi_n\colon Z_n\times \{1,\dots,n\}\to Y_n$) such that $\varphi_n(z,\cdot)\colon \{1,\dots,n\}\to f^{-1}(\{z\})$ (resp. $\psi_n(z,\cdot)\colon \{1,\dots,n\}\to g^{-1}(\{z\})$) is a Borel isomorphism for every $z\in Z_n$. Then arguing as in the $A_{\infty}$ cas, we may find Borel maps $h_i\colon A_n\to H$ (resp. $k_i\colon A_n\to H$), $i=1,\dots, n$, such that 
\[\varphi_n(\pi(x),i)=h_i(x)xH,\,\, \psi_n(\pi(x),i)=Hxk_i(x),\,x\in A_n,\,\,i=1,\dots,n.\]
Then $T_x=\{h_1(x),\dots,h_n(x)\}$ (resp. $S_x=\{k_1(x),\dots,k_n(x)\}$) is a transversal for $H\to H/(H\cap xHx^{-1})$ (resp. $H\to (H\cap x^{-1}Hx)\backslash H$). Thus the map $A_n\times \{1,\dots,n\}\ni (x,i)\mapsto h_i(x)xk_i(x)\in G$ is an injective Borel map, whence $T_n=\{h_i(x)xk_i(x)\mid i=1,\dots,n,x\in A_n\}$ is a Borel subset of $G$. Moreover, it is straightforward to see that 
\[\pi^{-1}(Z_n)=\bigsqcup_{g\in T_n}Hg=\bigsqcup_{g\in T_n}gH.\]
Therefore, $T=T_{\infty}\sqcup \bigsqcup_{n\in \N}T_n$ is a Borel set satisfying $G=\bigsqcup_{g\in T}Hg=\bigsqcup_{g\in T}gH$. 

\end{proof}

\section{Examples and applications}

Let now $G$ be a countable discrete group. We say that a group is maximally almost periodic (MAP) if the natural homomorphism from $G$ to its Bohr compactification $bG$ is injective, i.e., if and only if we can embed $G$ into a compact group. A subgroup $H \leq G$ is said to be Bohr closed if $\bar{H} \cap G = H$, where $\bar H$ denotes the closure of $H$ in $bG.$ Note that there is a continuous homomorphism from $bG$ to the pro-finite completion of $G$. In particular, every pro-finitely closed subgroup is also Bohr closed. Moreover, it is obvious that the intersection of Bohr closed subgroups is again Bohr closed. Note that the Bohr compactification is typically not metrizable. However, for the purposes of our arguments in this section, we always need  to separate only countably many elements at a time, so that we may pass to a metrizable quotient of $bG$ and apply our previous arguments there. 

\begin{lemma} \label{lem index}
Let $G$ be a countable, discrete group and let $K \leq H$ be Bohr closed subgroups. Then, $[H:K] = [\bar H: \bar K]$ if $[H:K]$ is finite and $[H :K] = \omega$ if and only if 
$[\bar H : \bar L]$ is infinite (necessarily uncountable). 
\end{lemma}
\begin{proof}
We write $H = \sqcup_{t \in T} tK$. If $tK \neq t'K$, then $t\bar K \neq t'\bar K,$ since otherwise $t^{-1}t' \in \bar K \cap G = K$ as $K$ is closed. We conclude that $\bar H \supset \sqcup_{t \in T} t \bar K$ with equality if $T$ is finite since the right-hand side is already closed. This proves the first part of the claim. The second part follows since the value $[\bar H : \bar K]$ is always uncountable whenever it is infinite.
This finishes the proof.
\end{proof}

\begin{theorem} \label{bohrtrans} Let $G$ be a countable, discrete group and let $H \leq G$ be a Bohr closed subgroup. Then there exists a common transversal for $H$.
\end{theorem}
\begin{proof} Without loss of generality, we may assume that $G$ is ${\rm MAP}$. Indeed, we may pass to the largest Bohr quotient of $G$, since any Bohr closed subgroup is pulled back from that quotient. 
Since $G$ is countable, thanks to its MAP property and the Bohr closedness of $H$, we may find a compact Polish group $\widehat{G}$ and a dense embedding $j\colon G\to \widehat{G}$ such that $\overline{H}\cap G=H$, where the closure is taken inside $\widehat{G}$. 


Now, for $x \in G$, it follows that $H \cap xHx^{-1}$ and $H \cap x^{-1}Hx$ are also Bohr closed. By Lemma \ref{lem index} the index of the inclusion of $H \cap xHx^{-1}$ in $H$ does not change after taking the closure if it is finite and it is $\omega$ if and only if it is $2^{\omega}$ after taking the closure since the closure is taken inside the metrizable compact group $\widehat{G}$. Now, Proposition \ref{prop same index} implies that the crucial equality $[H:H \cap xHx^{-1}] = [H:H \cap x^{-1}Hx]$ is always satisfied. Hence, there exists a common transversal for $H$ in $G$ using the decomposition $G = \sqcup_a HaH$ and the identifications $H \backslash HaH =  H/(H \cap aHa^{-1})$ and $HaH /H = (H \cap a^{-1}Ha) \backslash H$ as before.
\end{proof}

\begin{remark} In the arguments above, it was sufficient to consider metrizable quotients of $bG$. Let us emphasize that not all problems can be reduced to metrizable quotients so easily.
Note that if $G$ is {\rm MAP}, then $G$ is embedded as a dense subgroup in $bG$. This means there is always a net of elements in $G$ converging to the identity in $bG$. It is a surprisingly subtle question to decide, when there exists a {\it sequence} of non-trivial elements in $G$ that converges to the identity in $bG$, see \cite{MR3043070}.
\end{remark}

We end this section with a few examples and remarks.

\begin{example} \label{discrete}
Consider the Baumslag-Solitar group ${\rm BS}(1,2) = \langle a,b \mid bab^{-1}=a^2 \rangle$. It is well-known that ${\rm BS}(1,2)= \mathbb Z[\frac12] \rtimes \langle b \rangle$, where the element $b$ acts on $(\mathbb Z[\frac12],+)$ by multiplication with $2.$ Here, the element $a$ generates the standard copy of $\mathbb Z$ in $\mathbb Z[\frac12].$ Now, it is easy to see that $b \mathbb Z b^{-1} = 2 \mathbb Z \leq \mathbb Z[\frac12]$ and $b^{-1} \mathbb Z b = \frac12 \mathbb Z \leq \mathbb Z[\frac12].$ In particular, we get
$[\mathbb Z : \mathbb Z \cap b \mathbb Z b^{-1}]=2$ whereas $[\mathbb Z : \mathbb Z \cap b^{-1} \mathbb Z b]=1$. This implies that the number of left cosets and the number of right cosets contained in the double coset $\mathbb Zb\mathbb Z$ do not agree -- and hence there cannot be a common transversal for $\mathbb Z$ in ${\rm BS}(1,2)$. 

In view of the previous theorem, this is compatible with the fact that even though ${\rm BS}(1,2)$ is residually finite (and hence {\rm MAP}), see for example \cite[Lemma 2.4]{MR285589}, the subgroup $\mathbb Z$ is well-known not to be closed in the pro-finite topology. Indeed, $\mathbb Z$ is pro-finitely dense in $\mathbb Z[\frac12].$ On the other side, $\mathbb Z$ {\it is} Bohr closed as a subgroup of $\mathbb Z[\frac12]$, since $\mathbb Z[\frac12]/\mathbb Z$ is a subgroup of $S^1$ in a natural way. Now, the theorem above shows as a corollary that it is not Bohr closed in ${\rm BS}(1,2)$. This can also be shown by direct analysis of the finite-dimensional unitary representations of ${\rm BS}(1,2).$
\end{example}

\begin{example} \label{compact}
Consider the 2-solenoid ${\rm S}_2$, this can be defined as the Pontryagin dual of the discrete group $\mathbb Z[\frac12]$, and its subgroup $\mathbb Z_2$, the Pontryagin dual of $\mathbb Z[\frac12]/\mathbb Z$. Note that $\mathbb Z_2$ is just the group of 2-adic integers. Now, there is a crossed product ${\rm S}_2 \rtimes \mathbb Z$, where the generator of $\mathbb Z$ acts by multiplication with $2$. Indeed, this action can be defined on $\mathbb Z[\frac12]$ and thus has a continuous extension to its Pontryagin dual. This crossed product is a Polish group, whose connected component of the identity is just ${\rm S}_2$; in fact it is locally compact. Now, something similar happens as for the Baumslag-Solitar group, see Example \ref{discrete}. Indeed, $\mathbb Z_2 \cap b\mathbb Z_2b^{-1}$ has index $2$ in $\mathbb Z_2$ and  $\mathbb Z_2 \cap b^{-1}\mathbb Z_2b = \mathbb Z_2$, so again, the numbers of left and right cosets of $\mathbb Z_2$ in the double coset $\mathbb Z_2b\mathbb Z_2$ do not agree. Hence, there is no common transversal. And this happens even though $\mathbb Z_2$ is compact.

In view of our main result, we note that even though ${\rm S}_2$ is compact and an inverse limit of compact Lie groups, neither is the locally compact group ${\rm S}_2 \rtimes \mathbb Z$ an inverse limit of Lie groups nor is $\mathbb Z_2$ a compact Lie group. Whence, our main result does not apply.
\end{example}

\begin{remark}
According to the results in \cite{MR0442099AppelgateOnishi1977}, there exist closed common transversals for inclusions of pro-finite groups. However, note that this is not possible for general compact groups. Consider $S^1 \subset {\rm SU}(2)$ with homogeneous space ${\rm SU}(2)/S^1 = S^2.$ A closed common transversal would be homeomorphic to $S^2$ and hence ${\rm SU}(2)$ homeomorphic to $S^1 \times S^2$, which is not the case. Hence, there cannot be a closed common transversal - it is not hard to see that a locally closed transversal exists in this case. It remains an intriguing open problem to decide if in the context of Theorem \ref{thm main} a common transversal can be found with a specific rank in the Borel hierarchy; maybe even as a locally closed subset.
\end{remark}

\appendix

\section{The space of closed subsets of a Polish space}
In this appendix we summarize results on the hyperspace of closed (compact) subsets of a Polish space. Details an be found  e.g in  \cite{kechrisMR1321597,srivastavaMR1619545}. 
\subsection{The Effros Borel space $\mathcal{F}(X)$}
\begin{definition}
Let $X$ be a Polish space and $\mathcal{F}(X)$ the space of all closed subsets of $X$. The Effros Borel structure on $\mathcal{F}(X)$ is the $\sigma$-algebra generated by sets of the form 
\[\{F\in \mathcal{F}(X)\mid F\cap U\neq \emptyset\}\]
for an open set $U\subset X$. It is known that $\mathcal{F}(X)$ with the Effros Borel structure is a standard Borel space. 
We write $\mathcal{F}^*(X)=\mathcal{F}(X)\setminus \{\emptyset\}$ with the relative Borel structure.
\end{definition}
The proof of the next theorem can be found e.g., in \cite[Theorem 12.13]{kechrisMR1321597}. 
\begin{theorem}[Kuratowski--Ryll-Nardzewski]\label{thm Borel selector}
Let $X$ be a Polish space. There exists a sequence of Borel maps $\sigma_n\colon \mathcal{F}^*(X)\to X\,(n=1,2,\dots)$ such that $\{\sigma_n(F)\}_{n=1}^{\infty}$ is dense in $F$ for every $F\in \mathcal{F}^*(X)$.  
\end{theorem}
\subsection{Space of compact subsets $\mathcal{K}(X)$ and the Vietoris topology}
Let $X$ be a topological space, $\mathcal{K}(X)$ be the space of all compact subsets of $X$ with the Vietoris topology, i.e., the one generated by the sets of the form
\[\{K\in \mathcal{K}(X)\mid K\subset U\},\,\{K\in \mathcal{K}(X)\mid K\cap U\neq \emptyset\}\]
for $U$ open in $K$. A basis for the Vietoris topology then consists of the sets 
\[\{K\in \mathcal{K}(X)\mid K\subset U_0,\, K\cap U_i\neq \emptyset \,(i=1,\dots,n)\}\]
for $U_0,\dots,U_n$ open in $X$. 
If $X$ is metrizable with a compatible metric $d$, then the Hausdorff metric $\delta$ on $\mathcal{K}(X)$ with respect to $d$ is defined by 
\[\delta(K,L)=\begin{cases}\,\,0 & (K=L=\emptyset)\\ 
\,\,1 & (\text{exactly one of }K,L\text{ is }\emptyset,\\
\max\{\max_{x\in K}d(x,L),\max_{y\in L}d(y,K)\} & (K\neq \emptyset \neq L).
\end{cases}
\]
In this case, it can be shown that the Vietoris topology coincides with the metric topology given by the Hausdorff metric $\delta$ \cite[Proposition 2.4.14]{srivastavaMR1619545}. 
If $X$ is a compact metrizable space, then so is $\mathcal{K}(X)$ \cite[Proposition 2.4.17]{srivastavaMR1619545}. 
Since $\{\emptyset\}$ is isolated in $\mathcal{K}(X)$ (cf. \cite[Exercise 4.20]{kechrisMR1321597}), in this case the set $\mathcal{K}^*(X)=\{K\in \mathcal{K}(X)\mid K\neq \emptyset\}$ is also compact. 
\subsection{Wijsman topology on $\mathcal{F}^*(X)$}
Now if $X$ is a (not necessarily compact) Polish space with a complete compatible metric $d$, a Wijsman topology on $\mathcal{F}^*(X)$ is the weak topology generated by the maps $\mathcal{F}^*(X)\ni F\mapsto d(x,F)\in \R,\,x\in X$. The Wijsman topology is a Polish topology compatible with the Effros Borel structure \cite[$\S$4]{beerMR1065940}.
\subsection{$\mathcal{K}^*(X)$ as a Borel subspace of $\mathcal{F}^*(X)$}
We will use the fact that $\mathcal{K}^*(X)$ is a Borel subspace of $\mathcal{F}^*$ for the proof of Theorem \ref{thm main}. 
Since we were unable to find a proper reference, we record the proof here. 

\begin{lemma}\label{lem Wijsman=Vietoris}
Let $(X,d)$ be a Polish metric space. Then the inclusion map $j\colon \mathcal{K}^*(X)\to \mathcal{F}^*(X)$ is Vietoris--Wijsman continuous. In particular, the restriction of the Effros Borel structure of $\mathcal{F}^*(X)$ to $\mathcal{K}^*(X)$ agrees with the Borel structure induced by the Vietoris topology. If in addition $X$ is compact, then on $\mathcal{K}^*(X)=\mathcal{F}^*(X)$ the Vietoris topology and the Wijsman topology agree.  
\end{lemma}
\begin{proof}
We show that the identity map ${\rm id}\colon \mathcal{K}^*(X)\to \mathcal{F}^*(X)$ is sequentially continuous. 
Suppose $(K_n)_{n=1}^{\infty}$ is a sequence in $\mathcal{K}^*(X)$ which $\delta$-converges to some $K\in \mathcal{K}^*(X)$. 
Let $x\in X$. 
Assume by contradiction that $(d(x,K_n))_{n=1}^{\infty}$ does not converge to $d(x,K)$. Then there exists $\varepsilon>0$ such that at least one of the sets $I_+=\{n\in \N\mid d(x,K_n)\ge d(x,K)+\varepsilon\}$ or $I_-=\{n\in \N\mid d(x,K)\ge d(x,K_n)+\varepsilon\}$ is infinite. 
Assume that $I_+$ is infinite. By passing to a subsequence we may assume that $d(x,K_n)\ge d(x,K)+\varepsilon$ for all $n$. 
Let $y\in K$. Then $d(y,K_n)\le \max_{y'\in K}d(y',K_n)\le \delta(K,K_n)\xrightarrow{n\to \infty}0$. By compactness, for each $n\in \N$, there exists $y_n\in K_n$ such that $d(y,K_n)=d(y,y_n)$. 
Then 
\eqa{
d(y,y_n)&\ge d(x,y_n)-d(x,y)\\
&\ge d(x,K_n)-d(x,y)\\
&\ge d(x,K)+\varepsilon-d(x,y).
}
Let $n\to \infty$. Then we obtain 
\[d(x,y)\ge d(x,K)+\varepsilon.\]
Since $y\in K$ is arbitrary, it implies that 
\[d(x,K)\ge d(x,K)+\varepsilon,\]
which is a contradiction. 
Similarly, it is impossible that $I_-$ is infinite, whence $d(x,K_n)\xrightarrow{n\to \infty}d(x,K)$. Since $x$ is arbitrary, this shows that $K_n\xrightarrow{n\to \infty}K$ in $\mathcal{F}^*(X)$. Therefore, the inclusion map $j\colon \mathcal{K}^*(X)\to \mathcal{F}^*(X)$ is continuous. In particular, $j$ is an injective Borel map. Therefore it defines a Borel isomorphism of $\mathcal{K}^*(G)$ onto its image. If moreover $X$ is compact, then $j$ is a bijective continuous map from a compact space $\mathcal{K}^*(X)$ to a Hausdorff space $\mathcal{F}^*(X)$, whence it is a homeomorphism.   
\end{proof}

\section*{Acknowledgments}
A.T.\ thanks Henry Bradford for interesting discussions on the subject and Arno Fehm for the reference to \cite[Lemma 1.2.7]{MR2445111}.
H.A.\ is supported by Japan Society for the Promotion of Sciences KAKENHI 20K03647. 
\bibliography{references} 
\end{document}